\documentclass[11pt,a4paper]{article}
\usepackage[utf8x]{inputenc}
\usepackage[T1]{fontenc}
\usepackage[english]{babel} 
\usepackage[pdftex]{hyperref} 
\usepackage{enumitem} 
\usepackage[a4paper, lmargin=0.166\paperwidth, rmargin=0.166\paperwidth, tmargin=0.1111\paperheight, bmargin=0.1111\paperheight]{geometry}
\usepackage{thmtools}
\usepackage{hhline}
\usepackage{amsmath, amssymb, amsfonts}
\usepackage{csquotes}
\usepackage{amsthm}
\usepackage{cleveref}
\usepackage{xcolor}
\usepackage{tikz}

\title{\textbf{The Complexity of Connectedness Relations on Polish Spaces}}
\date{}
\author{Michal Hevessy\footnote{email: hevessy@karlin.mff.cuni.cz; Orcid: 0009-0003-4192-2407. This author was supported by the grant SVV-2025-260827 and by Czech Science Foundation grant 24-10705S.}\\
Department of Mathematical Analysis\\
Faculty of Mathematics and Physics, Charles University\\
Prague, Czechia\\
\and 
Yusuf Uyar\footnote{email: yuyar@metu.edu.tr; Orcid: 0009-0005-8082-2972. This author was supported by TÜBİTAK-BİDEB through 2210-E Program during his Master's degree education.}\\
Department of Mathematics\\
Faculty of Arts and Science, Middle East Technical University\\
Ankara, Turkey\\
\and 
Benjamin Vejnar\footnote{email: vejnar@karlin.mff.cuni.cz; Orcid: 0000-0002-2833-5385. This author was supported by Czech Science Foundation grant 24-10705S.}\\
Department of Mathematical Analysis\\
Faculty of Mathematics and Physics, Charles University\\
Prague, Czechia\\}

\theoremstyle{plain}
\newtheorem{Theorem}[]{Theorem}
\newtheorem{Proposition}[Theorem]{Proposition}
\newtheorem{Lemma}[Theorem]{Lemma}

\newtheorem*{Theorem*}{Theorem}
\newtheorem*{Lemma*}{Lemma}
\newtheorem*{Proposition*}{Proposition}
\newtheorem*{Question*}{Question}
\newtheorem{Question}{Question}

\newtheorem{Corollary}[Theorem]{Corollary}

\theoremstyle{definition}
\newtheorem*{Definition*}{Definition}
\newtheorem{Definition}[Theorem]{Definition}

\theoremstyle{remark}
\newtheorem{Remark}[Theorem]{Remark}
\newtheorem{Example}[Theorem]{Example}
\newtheorem*{Remark*}{Remark}

\newtheorem*{Example*}{Example}

\newcommand{\closure}[1]{\mkern 1.5mu\overline{\mkern-1.5mu#1\mkern-1.5mu}\mkern 1.5mu}
\makeatletter
\def\blfootnote{\gdef\@thefnmark{}\@footnotetext}
\makeatother

\DeclareMathOperator{\id}{id}

\DeclareMathOperator{\diam}{diam}

\newcommand{\Bleq}{\leq_B}

\newcommand{\Bsim}{\sim_B}

\newcommand{\Q}{\mathbb{Q}}
\newcommand{\R}{\mathbb{R}}
\newcommand{\N}{\mathbb{N}}

\begin{document}
\maketitle
\begin{abstract}
    We systematically investigate three different equivalence relations of connectedness: being connected by arcs, being connected by continua and being connected by chains of continua of decreasing diameter. The investigation is conducted from the point of view of Borel reductions, mainly on Polish spaces. All of the studied equivalence relations turn out to be  tied together and intimately related to the arc-connection relation. Among other results, it is shown that the arc-connection relation in the plane is Borel reducible to the Vitali equivalence relation and thus of a very low complexity. The same is proven for the chain continuum-connection relation on locally compact subsets of the plane, on which the continuum-connection relation is shown to have higher complexity.
\end{abstract}

\section{Introduction}
Let \(X\) be a metric space. We consider the following three equivalence relations
\begin{align*}
    E_X^A &= \{(x,y)\in X\times X: \text{ there is an arc in \(X\) containing \(x \text{ and }y\)}\} \\
    E_X^C &= \{(x,y)\in X\times X: \text{ there is a continuum in \(X\) containing \(x\text{ and }y\)}\} \\
    E_X^{CC} &= \{(x,y)\in X\times X: \text{ for all }  \varepsilon > 0, \text{ there are continua }K_0,...,K_n \subset X,n\in \N,\\  &\text{ of diameter }<\varepsilon \text{ such that }x\in K_0,y\in K_n \text{ and } K_i\cap K_{i+1}\neq \emptyset \text{ for } 0\leq i<n\}
\end{align*}
 on \(X\), which will be called \textit{arc-connection relation}, \textit{continuum-connection relation} and \textit{chain continuum-connection relation} (or simply, chain-connection relation); respectively.
All of these equivalence relations describe a notion of connectedness in a given space. We will observe that the continuum-connection and the chain-connection relations can be realized as arc-connection relations on some other related hyperspaces  (see \Cref{Theorem: continuum-connection is acrwise connecion on hyperspace} and  \Cref{Proposition: partial cain connection is arc-connection on hyperspace of small continua}). 

In this paper, we study the complexity of all three equivalence relations from the point of view of Borel reducibility. For this purpose, we restrict ourselves mainly to Polish spaces, i.e. completely metrizable separable topological spaces, although these are well-defined equivalence relations on any metric space. We will recall the basic notions of Borel complexity theory in the next section. On a Polish space, all of the three equivalence relations can be easily shown to be analytic. The arc-connection relation is the most studied one in the literature, and some of these works will be cited throughout the paper when appropriate.

In 2025, Debs and Saint Raymond showed that on a \(G_\delta\) subset of the plane the arc-connection relation is Borel \cite{Debs_Raymond_on_the_arc-wise_connection_in_plane}. This led to a very natural question formulated by the last author during the \(2025\) Winter School in Abstract Analysis: Is the arc-connection relation on compact subsets of the plane Borel reducible to the Vitali equivalence relation?

We answer this question more generally for any Polish subset of the plane. More precisely, we show that the arc-connection relation on planar Polish spaces can be Borel reduced to one of the simplest possible equivalence relations, namely the Vitali equivalence relation. We show that the same holds for the chain-connection relation on locally compact subsets of the plane and that it is not true for the continuum-connection relation. We also note that there are examples of continua in \(\R^3\) on which the arc-connection relation is non-Borel analytic \cite{Becker_number_of_path_components, Kunen_Starbird_arc_components_in_metric_continua, Donne_arc_components_in_metric_continua}. A summary of the existing results together with the new ones is shown in \Cref{table: summary of results}.

    \begin{table}[h]
        \centering
        \caption{A Summary of Results}
        \label{table: summary of results}
    {\renewcommand{\arraystretch}{1.25} 
    \resizebox{\textwidth}{!}{
    \begin{tabular}{|c||c|c|c|c|}
        \hline
        & Polish & locally compact Polish & \(G_\delta \subset \R^2\) & locally compact \(\subset \R^2 \)   \\
        \hhline{|=||=|=|=|=|}
        \(E_X^{A}\) & 
        \begin{tabular}{c}
             universal analytic\\
             \Cref{Remark: arc-connection in R^3 realizes all analytic equivalence relations}\\
             \cite{Becker_number_of_path_components}
        \end{tabular}
        & 
        \begin{tabular}{c}
             universal analytic   \\
             \Cref{Remark: arc-connection in R^3 realizes all analytic equivalence relations} \\
             \cite{Becker_number_of_path_components}
        \end{tabular}   
        & 
        \begin{tabular}{c}
             \(E_0\)  \\
             \Cref{Theorem: Arc-connection is E_0 in the plane} \\
             \Cref{Remark: arc-connection in Knaster continua are E_0}
        \end{tabular}
        & 
        \begin{tabular}{c}
             \(E_0\) \\
             \Cref{Theorem: Arc-connection is E_0 in the plane} \\
             \Cref{Remark: arc-connection in Knaster continua are E_0}
        \end{tabular} \\
        \hline
        \(E_X^C\) 
        & 
        \begin{tabular}{c}
             universal analytic \\
             \Cref{Remark: arc-connection in R^3 realizes all analytic equivalence relations}
        \end{tabular}
        & 
        \begin{tabular}{c}
             \(E_1\)  \\
             \Cref{Proposition: continuum-connection is hypersmooth for locally compact spaces} \\
             \Cref{Example: locally compact subset of the plane with continuum-connection E_1}
        \end{tabular}
        & 
        \begin{tabular}{c}
             \(\geq E_1\)  \\
             \Cref{Example: locally compact subset of the plane with continuum-connection E_1} \\
             Question \ref{Question: 1}, \ref{Question: 2}
        \end{tabular}
        & 
        \begin{tabular}{c}
             \(E_1\)  \\
             \Cref{Proposition: continuum-connection is hypersmooth for locally compact spaces} \\
             \Cref{Example: locally compact subset of the plane with continuum-connection E_1}
        \end{tabular}
        \\
        \hline
        \(E_X^{CC}\) 
        & 
        \begin{tabular}{c}
             universal analytic \\
             \Cref{Remark: arc-connection in R^3 realizes all analytic equivalence relations}             
        \end{tabular}
        &
        \begin{tabular}{c}
             \((\R/\ell_{\infty})^\N \)   \\
             \Cref{Theorem: chain continuum relation is smaller than l infinity to omega} \\
             \Cref{Example: X with chain continuum l infinity to omega}
        \end{tabular}
        
        & 
        \begin{tabular}{c}
             \(\geq E_0\) \\
             \Cref{Remark: arc-connection in Knaster continua are E_0} \\
             Question \ref{Question: 3}, \ref{Question: 4}
        \end{tabular}
        &
        \begin{tabular}{c}
             \(E_0\)   \\
             \Cref{Remark: arc-connection in Knaster continua are E_0}\\
             \Cref{Theorem: chain continuum-connection is E_0 in the plane on locally compact sets}
        \end{tabular}
        \\
        \hline
    \end{tabular}}
    }
\end{table}
Most of these results are the upper bounds on the corresponding spaces, and not only this but they are also optimal in the sense that we construct spaces that attain these upper bounds. The two exceptions are the continuum-connection and the chain-connection relations considered on \(G_\delta\) subsets of the plane. Although we construct some planar spaces whose continuum-connection and chain-connection relations are non-trivial, we do not know if these have to be even Borel on planar Polish spaces in general.

In \cite{Solecki_space_of_composants}, Solecki studied the composant equivalence relation on indecomposable continua. The continuum-connection relation can be seen as a generalization of this equivalence relation on Polish spaces. It is also worth noting that for some continua such as the Knaster continuum \cite[2.9]{Nadler_continuum_theory} the composant equivalence relation coincides with the arc-connection relation. 

It follows from the properties of Peano continua \cite[Chapter VIII]{Nadler_continuum_theory}  that the chain-connection relation can be seen as a generalization of the arc-connection relation. In fact, we will show that there are some very natural assumptions to ensure that the two coincide (Lemmas \ref{Lemma: arc-connection is the same as chain continuum-connection without triod}, \ref{Lemma: for a unicoherent space chain connection is arc-connection} and Corollary \ref{Corollary: for a space without riods chain connection is arc-connection}). 

Very closely related to the study of these equivalence relations is the study of the descriptive complexity of the set of subspaces that have one equivalence class of connection in the hyperspace of all compact subsets (or all subcontinua) of a Polish space. A great deal has been done concerning this in the literature for the arc-connection relation and for the more standard notions of connectedness as well (see \cite{Debs_SaintRaymond_complexity_of_connectivity, Debs_SaintRaymond_complexity_of_the_arc_connected_subsets_of_plane}). It is worth mentioning that the set of arc-connected subsets of the plane in the hyperspace of all compact subsets is of very high complexity as shown in \cite{Debs_SaintRaymond_complexity_of_the_arc_connected_subsets_of_plane}, which is in contrast with our result that the arc-connection relation in the plane itself is not very complicated. We refer the reader to \cite{marcone2006} for an extensive survey on the study of descriptive complexity of several notions of connectedness.

\section{Preliminaries}
In this section, we set the stage for studying the complexity of connectedness relations. We give basic definitions and examples about the theory of Borel equivalence relations. We then present necessary definitions and results from the continuum theory. Some properties of the hyperspaces of Polish spaces will be stated as well.

\subsection{Borel Equivalence Relations}
We use the notion of Borel reducibility to determine the complexity of equivalence relations. Further details regarding the theory of Borel equivalence relations and invariant descriptive set theory can be found, for instance, in \cite{Gao}.

A \textit{standard Borel space} is a measurable space which is isomorphic to the Borel $\sigma$-algebra of a Polish space, i.e. the $\sigma$-algebra generated by the open subsets of the space. We call measurable subsets of a standard Borel space \textit{Borel}. It turns out that any Borel subset of a Polish space is standard Borel. An equivalence relation $E$ on a standard Borel space $X$ is called a \textit{Borel equivalence relation} if $E$ is a Borel subset of the product space $X\times X$. A mapping between standard Borel spaces is a \textit{Borel mapping} if the inverse image of a Borel subset is Borel.

\begin{Definition}
    Let \(X,Y\) be standard Borel spaces and let \(E,F\) be Borel equivalence relations on \(X,Y\) respectively. We say that \(E\) is \textit{Borel reducible} to \(F\), denoted by \(E \Bleq F\), if there is a Borel mapping \(f: X \to Y\), called a \textit{Borel reduction}, such that 
    \[xEy \iff f(x)Ff(y)\] 
    for all \(x,y \in X\). We say that \(E\) is \textit{Borel bireducible} with \(F\), denoted by \(E \Bsim F\), if \(E \Bleq F\) and \(F \Bleq E\). We also write \(E <_B F\) whenever \(E \Bleq F\) and \(F \not\Bleq E\).
\end{Definition}
\begin{Remark*}
    Borel reductions compare in some sense the complexity of Borel equivalence relations. If we have \(E \Bleq F\), this intuitively means that the structure of \(E\) is captured in the structure of \(F\), and thus \(E\) is no more complex than \(F\). So it is a notion of relative complexity, sometimes referred as Borel complexity, on Borel equivalence relations.
\end{Remark*}
Our goal in general is to obtain Borel reductions from the connection relations to some well-understood Borel equivalence relations. Therefore, we need to recall some examples of Borel equivalence relations. We start with the simplest class of Borel equivalence relations.
\begin{Definition}
   Let \(X\) be a standard Borel space and \(E\) be a Borel equivalence relation on \(X\). We say that \(E\) is \textit{smooth}\ if \(E\) is Borel reducible to the identity relation on a standard Borel space.
\end{Definition}
There are several characterizations of smooth equivalence relations. For us, the following condition on Polish spaces will be enough. However, we note that it can be generalized in several ways.
\begin{Theorem}[{\cite[Theorem 6.4.4]{Gao}}]
\label{Theorem: equivalence relation with closed classes is smooth}
    Let \(X\) be a Polish space and \(E\) be a Borel equivalence relation on \(X\) such that every equivalence class of \(E\) is closed. Then \(E\) is smooth.
\end{Theorem}

Another very common equivalence relation that will prove to be quite useful for us is the following.

\begin{Definition}
    The equivalence relation \(E_0\) is defined by 
    \[xE_0y \iff \exists m\ \forall n \geq m\   x(n) = y(n).\]
    for all $x,y\in 2^\N $. It is sometimes called the \textit{Vitali equivalence relation}.
\end{Definition}
\begin{Remark*}
    It is easy to see that \(E_0\) can be written as a countable increasing union of Borel equivalence relations with all equivalence classes finite. Any Borel equivalence relation satisfying this condition is called \textit{hyperfinite}. It can be shown that \(E_0\) is the universal hyperfinite Borel equivalence relation in the sense that any hyperfinite Borel equivalence relation is Borel reducible to $E_0$ (see \cite[Section 7.2]{Gao}).
\end{Remark*}

The Borel equivalence relation \(E_0\) is known to be Borel bireducible with the Vitali equivalence relation on \(\R\) \cite[Proposition 6.1.4]{Gao}. It can be shown to be one of the simplest possible Borel equivalence relations in the following sense. Any Borel equivalence relation $E$ is either smooth or $E_0\leq _B E$. We need several results regarding \(E_0\). We start with a result concerning a decomposition of an equivalence relation into Borel parts each of which is Borel reducible to \(E_0\).
\begin{Lemma}
\label{Lemma: decomposition into hyperfinite parts}
    Let \(X\) be a standard Borel space and \(E\) be a Borel equivalence relation on \(X\). Suppose that \(X_n\) is an \(E\)-invariant Borel subset of \(X\) such that \(E|_{X_n} \leq_B E_0\) for all \(n\in \N \). Then \(E \leq_B E_0\).
\end{Lemma}
\begin{proof}
    Without loss of generality, we may assume that \(X_n\)'s are pairwise disjoint since otherwise we could replace $X_n$ by \(X_n \setminus \bigcup_{m < n}X_m\) for all $n\geq 1$. Then the statement follows from the fact that the Borel equivalence relation $\sim$ on the space \(\N \times 2^\N \) defined by \[(n,x) \sim (k,y) \iff n = k \text{ and } x E_0 y \] is Borel reducible to \(E_0\).
\end{proof}

The notion of treeability will also prove useful in the sequel. Recall that an equivalence relation is called \textit{countable} if all equivalence classes have at most countably many elements.
\begin{Definition}
        Let \(X\) be a standard Borel space and \(E\) be a Borel equivalence relation on \(X\). We say that \(E\) is \textit{treeable} if there exists a Borel acyclic graph \(G\), called a \textit{treeing} of \(E\), such that connected components of \(G\) coincide with $E$-classes.
\end{Definition}
Countable Borel equivalence relations with simple treeings are expected to have low Borel complexity, which is actually the case as stated below.
\begin{Proposition}[{\cite[Remark 6.8]{Kechris_Miller_topics_in_orbit_equivalence_relations}}]
\label{Proposition: treeing of hyperfinite quivalence}
 Let \(X\) be a standard Borel space and \(E\) be a countable Borel equivalence relation on \(X\). Suppose that \(G\) is a treeing of \(E\) such that the degree of \(G\) is \( 2\). Then \(E\leq _BE_0\).
\end{Proposition}

The following equivalence relation will come up very naturally when considering the continuum-connection relation on locally compact Polish spaces.

\begin{Definition}
    The equivalence relation \(E_1\) is defined by 
    \[xE_1y \iff \exists m\   \forall n \geq m\  \forall k\   x(n,k) = y(n,k)\]
    for all $x,y\in  2^{\N  \times \N }$.
\end{Definition}

\begin{Remark*}
A Borel equivalence relation that can be written as a countable increasing union of smooth equivalence relations is called \textit{hypersmooth}. The relation $E_1$ can be shown to be hypersmooth. Not only this, but $E_1$ is the universal hypersmooth equivalence relation (see \cite[Subsection 8.1]{Gao}).
\end{Remark*}
We introduce one more equivalence relation that will prove useful later.
\begin{Definition}
    The equivalence relation \(\R / \ell_\infty\)  is defined by
    \[x \R / \ell_\infty y \iff x-y \in \ell_\infty\]
    for all \(x,y\in \R^\N \).
\end{Definition}
This relation can again be shown to be a universal Borel equivalence relation on a special class of Borel equivalence relations.

\begin{Theorem}[{\cite[Theorem 8.4.2]{Gao}}]
\label{Theorem: l infinity is universal K sigma}
    Let \(E\) be a \(K_\sigma\) equivalence relation on a Polish space \(X\). Then \(E \Bleq \R / \ell_\infty\).
\end{Theorem}

All of the equivalence relations listed above are Borel. Moreover, it can be shown that they are of strictly increasing complexity.
\begin{Theorem}[{\cite[Proposition 6.1.7, Theorem 8.1.6, Section 15.2]{Gao}}]
    The equivalence relation \(E_0\) is not smooth and \(E_0 <_B E_1 <_B \R / \ell_\infty\).
\end{Theorem}

\subsection{Continuum Theory and Triods}
Throughout the paper, we will extensively work with continua, for which we need several results from continuum theory. We refer the reader to \cite{Nadler_continuum_theory} for any details and undefined notions.

Recall that a continuum \(K\) is called \textit{unicoherent} if the intersection of any two subcontinua of \(K\) is connected and \textit{hereditarily unicoherent} if every subcontinuum of \(K\) is unicoherent. In the sequel, we shall show that on some special classes of continua the studied equivalence relations have some nice properties. A continuum is called \textit{arc-like} (respectively \textit{circle-like}, \textit{tree-like}) if it is homeomorphic to an inverse limit of arcs (respectively circles, trees).

We recall the definitions of different types of triods and discuss the relations between them.
\begin{Definition}
   A \textit{ simple triod } is a continuum which is a union of three arcs intersecting at a single point.
\end{Definition}
\begin{Definition}
    A continuum \(K\) is called a \textit{triod} provided that there is \(L \in \mathcal{C}(K)\) such that \(K \setminus L\) has at least three components.
\end{Definition}
There is another notion of triod that is even less restrictive.
\begin{Definition}
    A continuum \(K\) is called a \textit{weak triod} if it is a union of three intersecting subcontinua, none of which is contained in the union of the other two.
\end{Definition}
Plainly, any simple triod is a triod. It turns out that even though the notion of weak triod is weaker than the notion of triod, any weak triod contains a triod.
\begin{Theorem}[{\cite[Theorem 1.8]{Sorgenfrey_triodic_continua}}]
\label{Theorem: weak triod contains triod}
    Every continuum containing a weak triod contains a triod.
\end{Theorem}

Triods are crucial notions that will come up very often in the study of all of the connection equivalence relations. The main reason why triods are of such importance is the following result of Moore.
\begin{Theorem}[{\cite[Theorem 1]{Moore_concerning_triods_in_the_plane}}]
\label{Theorem: pairwise disjoint family of triods in the plane is countable}
    Any family of pairwise disjoint triods in the plane is countable.
\end{Theorem}
This result is also the main reason why there are many essential differences between the behaviour of the connection equivalence relations in the plane and in spaces of higher dimensions.

The following properties of atriodic continua will be of some use as well. Recall that a continuum is called \textit{atrodic} if it does not contain a triod.

\begin{Lemma}[{\cite[Lemma 3.3]{The_hyperspace_for_triodic_continua}}]
\label{Lemma: triodic number of comp.}
    Let \(K\) be an atriodic continuum and let \(A,B \in \mathcal{C}(K)\). Then both \(B \setminus A\) and \(A \cap B\) have at most two components.
\end{Lemma}

\begin{Lemma}[{\cite[Lemma 3.5]{The_hyperspace_for_triodic_continua}}]
\label{Lemma: triodic complemet connected}
    Let \(K\) be an atriodic continuum and let \(A,B \in \mathcal{C}(K)\) be such that \(A \cap B\) is not connected. Then \(B \setminus A\) is connected.
\end{Lemma}

We will also use the following sufficient condition for a continuum to contain a triod.

\begin{Lemma}
\label{Lemma: condition for triod containing}
Let \(K\) be a continuum such that the following holds: There are two points \(x,y \in K\) and three pairwise distinct continua \(K_1,K_2,K_3 \subset K\) such that 
\begin{enumerate}
    \item \(K_i\) is irreducible between \(x \text{ and }y\) for \(i \in \{1,2,3\}\),
    \item \(K = \displaystyle K_1\cup K_2\cup K_3\).
\end{enumerate}
Then \(K\) contains a triod.
\end{Lemma}
\begin{proof}
    If for every distinct \(i,j,k\) we have \(K_i \not\subset K_j \cup K_k \), then \(K\) is a weak triod since \( K_1\cap K_2\cap K_3\neq \emptyset\). Thus, \(K\) contains a triod by \Cref{Theorem: weak triod contains triod}. So we may assume without loss of generality that \(K_1 \subset K_2 \cup K_3\).

     Assume for the sake of contradiction that \(K\) does not contain a triod. Let \(X = K_2 \setminus K_3\) and \(Y = K_3 \setminus K_2\). We note that \(X\) and \(Y\) are both open in \(K\). We consider two different cases.
    
    First, suppose that there are some points \(a \in X \setminus K_1 \) and \(b \in Y \setminus K_1\). Then we can find open neighbourhoods \(U_a\) of \(a\) and \(U_b\) of \(b\) in \(K\) such that \(\closure{U_a} \subset X\) with \(\closure{U_a} \cap K_1 = \closure{U_a} \cap K_3 = \emptyset\) and \(\closure{U_b} \subset Y\) with \(\closure{U_b} \cap K_1 = \closure{U_b} \cap K_2 = \emptyset\).
    Consider \(C_x^a\) the component of \(x\) in the set \(\closure{(K_2\setminus \closure{U_a})}\) and \(C_x^b\) the component of \(x\) in the set \(\closure{(K_3\setminus \closure{U_b})}\). Then from the boundary bumping theorem we have that \(C_x^a \cap \partial \closure{U_a} \neq \emptyset\), \(C_x^a \cap \partial \closure{U_b} = \emptyset\), \(C_x^b \cap \partial \closure{U_b} \neq \emptyset\) and \(C_x^b \cap \partial \closure{U_a} = \emptyset\). We note that \(y \notin C_x^a \cup C_x^b\). Since \(K_1 \cap \partial\closure{U_a} = K_1 \cap \partial\closure{U_b} = \emptyset\) and \(x \in K_1 \cap C_x^a \cap C_x^b\) we get that \(K_1 \cup C_x^a \cup C_x^b\) is a weak triod. Thus by \Cref{Theorem: weak triod contains triod} \(K\) contains  triod.

    Second, suppose that \(X \subset K_1\) or \(Y \subset K_1\). Without loss of generality we assume that \(X \subset K_1\). By irreducibility, \(K_2\cap K_3\) must be disconnected, and so it has two components by Lemma \ref{Lemma: triodic number of comp.}, call \(A\) and \(B\). Say $x\in A$ and $y\in B$.  
   
From \Cref{Lemma: triodic complemet connected} we have that both \(X\) and \(Y\) are connected. Since \(A \cup B\) is disconnected and \(K_2 = A \cup B \cup X\), we also get that \(X \cup A\) and \(X \cup B\) are both connected. By \Cref{Lemma: triodic number of comp.} we have that \(K_1 \cap K_2\) has two components, call \(C_1\) and \(C_2\). Without loss of generality we may assume that \(X \subset C_1\) and \(x \in C_1\). Since \(K_1\) is irreducible between \(x\) and \(y\), we have that \(y \not \in C_1\). Since \(X \subset C_1\) and \(X \cup B \) is connected, we have that \(C_1 \cup B \subset K_2\) is connected and so \(C_1 \cap B \neq \emptyset\). We note that \(K_1 \cap Y \neq \emptyset\) because otherwise we would have \(K_2 \subset K_1\), which is not possible. Let \(c \in C_1 \cap B\) and let \(C_c\) be a component of \(c\) in \(K_1 \cap K_3\). Then we have that \(C_c \cap Y \neq \emptyset\). We note that \(B \not\subset K_1\), otherwise \(C_1 \cup B\) would be a proper subcontinuum of \(K_1\) containing \(x\) and \(y\). Then the space \(C_1 \cup B \cup C_c\) is a weak triod. Thus, \(K\) contains a triod by \Cref{Theorem: weak triod contains triod}.
\end{proof}

We need one more fact regarding the hyperspace of continua.

\begin{Theorem}[{\cite[Theorem 14.9]{Nadler_hyperspaces}}]
\label{Theorem: hyperspace of continuum is arc-wise connected}
    Let \(K\) be a continuum. Then the space \(\mathcal{C}(K)\) is arc-wise connected.
\end{Theorem}

\subsection{Preliminary Results}
In this subsection, we present earlier results on the hyperspaces of Polish spaces and the properties of Borel structure on these hyperspaces that will be used throughout the paper. We refer the reader to \cite{Kechris} for further details.

Let \(X\) be a Polish space. We denote the hyperspace of non-empty compact subsets of $X$ and the hyperspace of continua in $X$ by  \(\mathcal{K}(X)\) and  \(\mathcal{C}(X)\); respectively. We endow $\mathcal{K}(X)$ with the Vietoris topology, which can be induced by the Hausdorff metric on it. Recall that this makes $\mathcal{K}(X)$ Polish and that the space \(\mathcal{C}(X)\) is a closed subset of \(\mathcal{K}(X)\). We also note that the Effros-Borel structure on the hyperspace \(\mathcal{F}(X)\) of all non-empty closed subsets of \(X\) is standard Borel. Moreover, the set \(\mathcal{K}(X)\), and therefore \(\mathcal{C}(X) \), is a Borel subset of \(\mathcal{F}(X)\).

One of the most frequently used results throughout the paper is the Lusin-Novikov uniformization theorem.

\begin{Theorem}[{\cite[Theorem 18.10]{Kechris}} Lusin-Novikov]
\label{Theorem: Lusin-Novikov}
    Let \(X,Y\) be standard Borel spaces and let \(P \subset X \times Y\) be Borel. If every section \(P_x\), \(x \in X\), of \(P\) is countable, then \(P\) has a Borel uniformization and therefore the projection \(\pi_1( P)\) is Borel.
\end{Theorem}
Another useful result is the existence of a Borel selector function on closed subsets.
\begin{Theorem}[{\cite[Theorem 12.13]{Kechris}}]
\label{Theorem: selection for closed sets}
Let \(X\) be a Polish space. Then there exists a Borel function \(S: \mathcal{F}(X) \to X\) such that \(S(H) \in H\) for every \(H \in \mathcal{F}(X)\).
\end{Theorem}

While working with the arc-connection relation, it is naturally very useful to consider the set of all arcs in the hyperspace \(\mathcal{K}(X)\). We denote the set of all arcs and the set of all simple closed curves in $X$ by \(\mathcal{J}(X)\) and \(\mathcal{S}(X)\); respectively. It turns out that both subsets are Borel. Not only this, but we also have the following properties.
\begin{Proposition}[{\cite{Becker_note_on_path_components}}]
\label{Proposition: arcs and circles are Borel}
    Let \(X\) be a Polish space. Then \(\mathcal{J}(X)\) and \(\mathcal{S}(X)\) are Borel subsets of \(\mathcal{K}(X)\), and there exists a Borel mapping \(e: \mathcal{J}(X) \to \mathcal{K}(X)\) that assigns to every arc the set of its endpoints.
\end{Proposition}

This can be further improved by splitting the mapping \(e\) into two mappings.
\begin{Lemma}[\cite{Debs_Raymond_on_the_arc-wise_connection_in_plane}]
\label{Lemma: Borel assignemtns of endpoints}
    Let \(X\) be a Polish space. Then there are two Borel mappings \(e_1,e_2:\mathcal{J}(X) \to X\) such that for every \(I \in \mathcal{J}(X)\) we have \(e(I) = \{e_1(I),e_2(I)\}\), where \(e\) is the mapping in \Cref{Proposition: arcs and circles are Borel}.
\end{Lemma}

As mentioned earlier, an important result which sparked this investigation is that the arc-connection relation on a \(G_\delta\) subset of the plane is Borel, where we recall that $G_\delta$ subsets of a Polish space are exactly Polish subspaces of it.
\begin{Theorem}[{\cite[Theorem 1]{Debs_Raymond_on_the_arc-wise_connection_in_plane}}]
\label{Theorem: arc-connection relation is borel for G_delta in the plane}
    Let \(X\subset \mathbb{R}^2\) be a \(G_\delta\) subset. Then \(E_X^A\) is Borel.
\end{Theorem} 

\section{Arc-connection Relation}
In this section, we investigate the arc-connection relation. Let \(X\) be a metric space and let \(C\) be an arc component of \(X\). Then there are three possible forms for $C$, see \cite{Debs_Raymond_on_the_arc-wise_connection_in_plane}. The first possibility is that \(C\) contains only one point. The second possibility is that \(C\) is a continuous image of an interval (closed, open or half-open), in which case it is a \(K_\sigma\) and thus a Borel subset. The last possibility is that \(C\) contains a simple triod, in which case $C$ is called  a \textit{triodic arc component}. We will show that in some sense the complexity of arc-connection relation is determined by the complexity of its non-triodic non-trivial arc components in planar spaces, whereas it is determined by the complexity of its triodic arc components in a general Polish space. Debs and Saint Raymond showed that Borelness of the other arc components is ensured if the triodic arc components are Borel.
\begin{Theorem}[{\cite[Proposition 10]{Debs_Raymond_on_the_arc-wise_connection_in_plane}}]
\label{Theorem: borelness of isolated points}
    Let \(X\) be a Polish space. Suppose that the union of all points with triodic arc components is Borel. Then the set of all points with trivial arc components is Borel as well.
\end{Theorem}
The main result of this section is that on a \(G_\delta\) subset of the plane the arc-connection relation is Borel reducible to \(E_0\). We start with a more general theorem. We can then use the results of \cite{Debs_Raymond_on_the_arc-wise_connection_in_plane} to apply the theorem to the planar case. We also demonstrate an application of this theorem in higher dimensions.

\begin{Theorem}
\label{Theorem: general_theorem}
    Let \(X\) be a Polish space and let 
    \[X_T = \{x \in X: \text{ \(x\) has a triodic arc component}\}.\]
    Suppose that \(X_T\) and \(E^A_X\) are Borel. Then \(E^A_X \leq_B E_0 \times E^A_X|_{X_T}\).
\end{Theorem}
\begin{proof}
    Since any Polish space can be embedded into the Hilbert space, we may assume that  \(X\subset \ell_2(\N)\). Let 
    \[X_P = \{x \in X: \text{ \(x\) has trivial arc component}\}\] 
    and set \(X_N = X \setminus (X_T \cup X_P)\). By \Cref{Theorem: borelness of isolated points}, all the sets \(X_T,X_P,X_N\) are Borel sets and they are \(E^A_X\)-invariant. We note that any non-triodic arc component that contains a simple closed curve is itself a simple closed curve. Consider 
    \[C = \{(x,S) \in X_N \times \mathcal{S}(X): x \in S\},\]
    where \(\mathcal{S}(X)\) is the set of all simple closed curves in \(X\). The set \(C\) is Borel as \(\mathcal{S}(X)\) is Borel by \Cref{Proposition: arcs and circles are Borel}. For every \(x \in X_N\) we have that the section \(C_x\) can contain at most one simple closed curve, and therefore we get that \(X_S = \pi_1(C)\) is Borel  by \Cref{Theorem: Lusin-Novikov}. Denote \(X_A = X_N \setminus X_S\). We have that \(X_A\) is an \(E^A_X\)-invariant Borel set, and it is the set of all points whose arc component is non-trivial and does not contain any triod or a simple closed curve.

    First, we show that \(E^A_X |_{X_A} \leq_B E _0\). Let \(\mathcal{J}(X)\) be the set of all arcs in \(X\). The set \(\mathcal{J}(X)\) of all arcs in \(X\) is Borel by \Cref{Proposition: arcs and circles are Borel} and the maps \(e,e_1,e_2\) on \(\mathcal{J}(X)\) defined as in \Cref{Lemma: Borel assignemtns of endpoints} are Borel as well. We define an auxiliary equivalence relation \(E_A\) on \(\mathcal{J}(X)\) of lying in the same arc component as 
    \[E_A = \{(I,J) \in \mathcal{J}(X) \times \mathcal{J}(X): e_1(I) E^A_Xe_1(J)\}.\]
    Since \(E^A_X\) is Borel, then so is \(E_A\).
    For any \(n \in \N\) and \(p \in \R\) consider the set
    \[S(n,p) = \{(x_i)_{i \in \N} \in \ell_2(\N): x(n) = p\},\]
    which is a closed subset of \(\ell_2(\N)\). For any distinct real numbers $p$ and $r$, consider the Borel set $\mathcal{J}(n,p,r)$ given by
$$\bigcup _{\{i,j\}=\{1,2\}}\{J \in \mathcal{J}(X): S(n,p) \cap J = \{e_i(J)\}, S(n,r) \cap J = \{e_j(J)\}, e_1(J) \in X_A \}.$$
This is the set of all arcs in $X_A$ 
that intersect the disjoint closed sets $S(n,p)$ and $S(n,r)$ only at their endpoints. We will use \Cref{Proposition: treeing of hyperfinite quivalence} to show that \(E_A |_{\mathcal{J}(n,p,r)}\) is Borel reducible to \(E_0\) for any \(n \in \N\) and any distinct \(p,r \in \R\). Fix arbitrary \(n \in \N \) and distinct \(p,r \in \R\). Consider the Borel set $D=D_1\cap D_2$ where 
$$D_1=\{(I,J,\Gamma) \in \mathcal{J}(n,p,r)^2 \times \mathcal{J}(X): I \neq J, e(\Gamma) \subset e(I) \cup e(J)\} \text{ and }$$
$$D_2=\{(I,J,\Gamma) \in \mathcal{J}(n,p,r)^2 \times \mathcal{J}(X): \Gamma \cap S(n,p) = \emptyset \text{ or } \Gamma \cap S(n,r) = \emptyset \text{ or } \Gamma = I \cup J\} .$$
Since the set \(\mathcal{J}(n,p,r)\) contains only arcs from non-triodic arc components without simple closed curves, the section $D_{(I,J)}$ contains at most one arc for any pair \((I,J)\) of arcs. Hence the set \( E = \pi_{1,2}(D)\) is Borel by \Cref{Theorem: Lusin-Novikov}. We will show that \(G = (\mathcal{J}(n,p,r),E)\) is a Borel treeing of \(E_A|_{\mathcal{J}(n,p,r)}\) that satisfies the assumptions of \Cref{Proposition: treeing of hyperfinite quivalence}.

    The fact that \(G\) is a graph follows from the construction. We need to show that connected components of the graph \(G\) are exactly equivalence classes of \(E_A |_{\mathcal{J}(n,p,r)}\). Pick \((I,J) \in E_A |_{\mathcal{J}(n,p,r)}\). Let $K$ be the smallest arc containing both $I$ and $J$, which clearly exists as we are in a non-triodic arc component without simple closed curve. Since \(K\) is compact, there are at most finitely many arcs from \(\mathcal{J}(n,p,r)\) lying in \(K\). Say \[\{I=I_0, I_1, \dots, I_k =J\}\subset \mathcal{J}(n,p,r)\] is the set of arcs lying in $K$ in the listed order. Let $0\leq i<k$ and let $K'$ be the smallest arc containing $I_i$ and $I_{i+1}$. Then  $(I_i,I_{i+1},K')\in D$ because $K'$ does not contain an element of $\mathcal{J}(n,p,r)$ other than $I_i$ and $I_{i+1}$. It follows that $(I_i,I_{i+1})\in E$. This argument shows that $I$ and $J$ are in the same connected component of the graph $G$. Conversely, if \((I,J)\in E\), then it is clear from the definition of \(G\) that \((I,J) \in E_A |_{\mathcal{J}(n,p,r)}\).
    
    Since we are considering only non-triodic arc components without simple closed curves, we get that \(G\) is acyclic and has degree at most \(2\). It follows then that \(E_A|_{\mathcal{J}(n,p,r)}\) is countable and hence \(E_A|_{\mathcal{J}(n,p,r)} \leq_B E_0\) by \Cref{Proposition: treeing of hyperfinite quivalence}. Let 
    \[
    M(n,p,r) = \{(x,J) \in X_A \times \mathcal{J}(n,p,r): x E_Ae_1(J)\},
    \]
    which is a Borel set. Since \(E_A|_{\mathcal{J}(n,p,r)}\) is countable, we get that for any \(x \in X_A\) the section \(M(n,p,r)_x\) is countable.  Then   there is a Borel uniformization of \(M(n,p,r)\) by \Cref{Theorem: Lusin-Novikov}. Moreover, this uniformization is a Borel reduction from \(E_X|_{\pi_1(M(n,p,r))}\) to \(E_A|_{\mathcal{J}(n,p,r)}\), implying that \(E|_{\pi_1(M(n,p,r))} \leq_B E_0\). Clearly, for any \(n \in \N\) and distinct \(p,r \in \R\) we have that \(\pi_1(M(n,p,r))\) is \(E^A_X\)-invariant. It is easy to see that 
    \[
    X_A=\bigcup_{\substack{p\neq q \in \Q \\
    n \in \N }}\pi_1\left( M(n,p,r)\right).
    \]
We also have that $E^A_X|_{X_S}\leq _B E_0$ for which the set $C$ defined earlier is a witness and that $E^A_X|_{X_P}$ is itself the equality relation. Combining all parts together, we get that
    \[\{\pi_1(M(n,p,r)): \ p,q \in \Q, \ p\neq r \text{ and }n \in \N \} \cup \{X_S,X_P\}\]
    is a collection as in \Cref{Lemma: decomposition into hyperfinite parts} and so \(E^A_X|_{X_N \cup X_P} \leq_B E_0\). Since \(X_T\) and \(X_N \cup X_P\) are \(E^A_X\)-invariant and disjoint, we get
    \begin{equation*}
        E_X \leq_B E_X|_{X_N \cup X_P} \times E_X |_{X_T}  \leq_B E_0 \times E_X |_{X_T}\ . \qedhere
    \end{equation*}
\end{proof}
\begin{Remark*}
    \Cref{Theorem: general_theorem} intuitively says that any complex behaviour of \(E_X^A\) on a Polish space \(X\) must arise from the triodic arc components.
\end{Remark*}
Now, we can use this theorem to prove that the arc-connection relation in the plane is Borel reducible to \(E_0\).

\begin{Theorem}
\label{Theorem: Arc-connection is E_0 in the plane}
    Let \(X\) be a \(G_\delta\) subset of the plane. Then \(E^A_X \Bleq E_0\).
\end{Theorem}
\begin{proof}
     By \Cref{Theorem: arc-connection relation is borel for G_delta in the plane}, we have that \(E^A_X\) is Borel. Since there are only countably many triodic arc components, the set 
    \[X_T = \{x\in X: x \text{ has a triodic arc component} \}\]
    is Borel as well. Having countably many equivalence  classes, \(E^A_X|_{X_T} \leq_B \id (\N )\). Thus,
    
     \[E^A_X \leq_B E_0 \times E^A_X|_{X_T}\leq _BE_0\times \id (\N )\leq _BE_0,\] where the first Borel reduction is by \Cref{Theorem: general_theorem} and the last one is trivial.
\end{proof}
\begin{Remark}
\label{Remark: arc-connection in Knaster continua are E_0}
Let \(K\) be the Knaster continuum (see \cite[2.9]{Nadler_continuum_theory}). This continuum is also called the bucket handle continuum by some authors. From the properties of \(K\), we have that the composant equivalence relation on \(K\) coincides with \(E^A_K\), and thus we get \(E_0 \Bleq E_{K}^A\)  as shown in \cite{Solecki_space_of_composants}. So we conclude by \Cref{Theorem: Arc-connection is E_0 in the plane} that \(E_K^A \Bsim E_0\). This can also be shown directly as was done in \cite{uyar2025nonsmoothpathconnectednessrelationreal}. Borel bireducibility of \(E_K^A\) with \(E_0\) can in fact be shown for any of the generalized Knaster continua (see \cite{Watkins_certain_inverse_limits}).
\end{Remark}

As mentioned in the introduction, there are examples of continua \(X\subset \R^3\) for which \(E_X^A\) is analytic non-Borel \cite{Becker_number_of_path_components, Kunen_Starbird_arc_components_in_metric_continua, Donne_arc_components_in_metric_continua} and so the assumption of embeddability in the plane cannot be relaxed. A natural question is whether the descriptive complexity assumption that $X\subset \mathbb{R}^2$ is a $G_\delta$ subset can be weakened or not. We will show later that the answer to this question is negative. That is, we shall construct an \(F_\sigma\) subset $X$ of the plane for which \(E_X^A\) is not even an analytic equivalence relation.
\begin{Remark}
\label{Remark: arc-connection in R^3 realizes all analytic equivalence relations}
    Becker \cite[Theorem 4.1]{Becker_number_of_path_components} states that any analytic equivalence relation can be realized as an arc-connection relation on a compact subset of \(\R^3\). Becker does not include a complete proof; however, it can be found in \cite[Theorem 3.2]{wu2025fundamental}. The construction can, in fact, be easily modified to yield a Polish space on which the arc-connection, continuum-connection and chain-connection relations are the same.
\end{Remark}

We end this section by two simple results. One is an application of \Cref{Theorem: general_theorem} to the class of circle-like continua and the other is the fact that the arc-connection relation on the product space essentially coincides with the product of the arc-connection relation on the space. We will show that the latter property holds for all of the studied connection relations.
\begin{Corollary}
    Let \(X\) be a circle-like continuum. Then \(E_X^A \leq E_0\).
\end{Corollary}
\begin{proof}
    Since \(X\) is circle-like, it does not contain any triod. Thus, the result follows directly by \Cref{Theorem: general_theorem}.
\end{proof}
\begin{Proposition}
    Let \(X\) be a Polish space. Then \(E^A_{X^\N }\) is Borel isomorphic to \((E_X^A)^\N \), and so \(E^A_{X^\N } \Bsim (E_X^A)^\N \).
\end{Proposition}
\begin{proof}
    This follows easily since two points \((x_n)_{n \in \N }, (y_n)_{n \in \N } \in X^\N \) are arc-wise connected in \(X^\N \) if and only if \(x_n \text{ and }y_n\) are arc-wise connected in \(X\) for every \(n \in \N \).
\end{proof}

\section{Continuum-connection Relation}
The continuum-connection relation can be viewed as a possible generalization of the composant equivalence relation studied in \cite{Solecki_space_of_composants}. It can be easily seen that the continuum-connection relation on a compact space is smooth because the space \(\mathcal{C}(X)\) is itself compact whenever $X$ is. On a continuum, this relation has only one equivalence class.

It turns out that the continuum-connection relation is very closely related to the arc-connection relation. In fact, the continuum-connection relation is essentially the same as the arc-connection relation on the corresponding hyperspace of continua as shown below.
\begin{Theorem}
\label{Theorem: continuum-connection is acrwise connecion on hyperspace}
    Let \(X\) be a Polish space. Then \(E^C_X \Bsim E^A_{\mathcal{C}(X)}\).
\end{Theorem}
\begin{proof}
    \((\leq _B)\) Consider the mapping \(f:X \to \mathcal{C}(X)\) given by \(x \mapsto \{x\}\). This mapping is continuous and hence Borel. We will show that it is the desired Borel reduction. Let \((x,y) \in E_X^C\). Then there exists \(K \in \mathcal{C}(X)\) such that \(x,y \in K\). By \Cref{Theorem: hyperspace of continuum is arc-wise connected}, we have that the space \(\mathcal{C}(K)\) is arc-wise connected. So there is an arc \(I \subset \mathcal{C}(K)\) from \(\{x\}\text{ to }\{y\}\). Since \(\mathcal{C}(K) \subset \mathcal{C}(X)\), we have that \(\{x\}\) and \(\{y\}\) are connected by an arc in \(\mathcal{C}(X)\) as well. For the converse implication, suppose that \((\{x\},\{y\}) \in E_{\mathcal{C}(X)}^A\). Then there is a continuous mapping \(\varphi: [0,1] \to \mathcal{C}(X)\) with \(\varphi(0)=\{x\} \text{ and } \varphi(1)=\{y\}\).  Since $[0,1]$ is a continuum and $\varphi$ together with the union map are continuous, we conclude that \(\bigcup_{t \in [0,1]}\varphi(t)\) is a continuum. Thus, $(x,y)\in E^C_X$.

   \((\geq _B)\) Consider a Borel selector function \(S:\mathcal{C}(X)\rightarrow X\) as in \Cref{Theorem: selection for closed sets}. We will show that \(S\) is the desired Borel reduction. Let \((K,L) \in E_{\mathcal{C}(X)}^A\). Since \(\mathcal{C}(K)\) and \(\mathcal{C}(L)\) are arc-wise connected, we can find an arc in \(\mathcal{C}(X)\) connecting \(\{S(K)\}\) and \(\{S(L)\}\). This implies, using a similar argument as before, that \(S(K)\) and \(S(L)\) are contained in a continuum in \(X\). The other implication also follows by the same reasoning as in the first Borel reduction above. 
\end{proof}
The properties of the continuum-connection relation on compact spaces can be used to show that the continuum-connection is hypersmooth on locally compact Polish spaces.
\begin{Proposition}
\label{Proposition: continuum-connection is hypersmooth for locally compact spaces}
    Let \(X\) be a locally compact Polish space. Then \(E^C_X \Bleq E_1\).
\end{Proposition}
\begin{proof}
    Since \(X\) is locally compact and separable, we can find an increasing sequence \((G_n)_{n \in \N }\) of open sets such that \(\closure{G_n}\) is compact and \(X = \bigcup_{n \in \N }G_n\). Define \[E_n = E^C_{\closure{G_n}} \cup 
    \id(X)\]
    for any \(n\in \N\). It is easy to see that \(E_n\) is an equivalence relation on \(X\) for every \(n \in \N \). We note that since \(\closure{G_n}\) is compact, any two points of \(\closure{G_n}\) are connected by a continuum in \(\closure{G_n}\) if and only if they are in the same connected component of \(\closure{G_n}\) for all \(n \in \N \). In particular, the equivalence classes of \(E_n\) are points or connected components of \(\closure{G_n}\) for all $n\in \N $. Thus, $E_n$ has closed equivalence classes in \(X\) and by \Cref{Theorem: equivalence relation with closed classes is smooth} we get that \(E_n\) is smooth for every \(n \in \N \). We will show that \(E^C_X = \bigcup_{n\in \N } E_n\).

    The fact that \( \bigcup_{n \in \N }E_n\subset E_X^C \) follows easily since \(E_n \subset E_X^C\) for every \(n \in \N \). For the other inclusion, pick \((x,y) \in E_X^C\). Then there exists \(K \in \mathcal{C}(X)\) such that \(x,y \in K\). Since \(K\) is compact, \(K \subset G_n\) for some $n_0\in\N $. Thus \((x,y) \in E_{n_0}\subseteq \bigcup_{n \in \N }E_n\). This finishes the proof that $E^C_X$ is hypersmooth.
\end{proof}
It turns out that even in the plane the continuum-connection relation can be strictly more complicated than the arc-connection relation.

\begin{Example}
\label{Example: locally compact subset of the plane with continuum-connection E_1}
    There is a locally compact subset $X$ of the plane such that \(E_X^C \Bsim E_1\).
\end{Example}
\begin{proof}
 Let \(\mathbb{P}\) be the pseudo-arc. In \cite{Solecki_space_of_composants}, Solecki showed that the composant equivalence relation on \(\mathbb{P}\), call $F$, is Borel bireducible with \(E_1\). Denote this equivalence relation by \(F\). Let \(x \in \mathbb{P}\) and consider \(X = \mathbb{P} \setminus \{x\}\). By \Cref{Proposition: continuum-connection is hypersmooth for locally compact spaces}, we have \(E_X^{C} \Bleq E_1\). Thus, it is enough to argue that \(E_1 \Bleq E_X^{C}\). We will show that \(F \Bleq E_X^C\). To this end, let \(C\) be the composant of \(\mathbb{P}\) containing \(x\). Now pick \(y \in C \setminus \{x\}\) and consider the function \(f: \mathbb{P} \to X \) defined by
 \[
 f(p) = \begin{cases}
     p, \quad& p \in \mathbb{P} \setminus C \\
     y, \quad&p \in C.
 \end{cases}
 \]
 Then \(f\) is a Borel reduction of \(F\) to \(E_X^{C}\) and therefore \(F \Bleq E_X^{C}\).
\end{proof}

To prove that the arc-connection relation on $G_\delta$ subsets of the plane is Borel, Debs and Saint Raymond used the result by Becker and Pol \cite[Proposition 5.1]{Becker_note_on_path_components} that the arc-connection relation on such spaces is Borel if and only if all of the equivalence classes of the relation are Borel. We show below that the same holds for the continuum-connection relation in the plane. However, this equivalence of conditions is not true in general for analytic equivalence relations. 

\begin{Theorem}
\label{Theorem: continuum-connection is Borel iff all classes are Borel}
    Let \(X\) be a \(G_\delta\) subset of the plane. Then \(E_X^C\) is Borel if and only if every equivalence class of \(E_X^C\) is Borel.
\end{Theorem}
\begin{proof}
   Each equivalence class of any Borel equivalence relation is clearly Borel. For the other implication, assume that every equivalence class of \(E_X^{C}\) is Borel. Let 
    \[A = \{(x,y,K) \in X^2 \times \mathcal{C}(X): x,y \in K,\ K \text{ is irreducible between \(x\) and \(y\)}\}.\] We will show that this set is Borel. Consider the set 
    \[B = \{(x,y,K,L) \in X^2 \times \mathcal{C}(X)^2: x,y \in K, \ x,y \in L \text{ and } \ L \subsetneq K\},\] 
    which can be easily seen to be Borel. For any \((x,y,K) \in X^2 \times \mathcal{C}(X)\), the section \begin{align*}
        B_{(x,y,K)} = \{L \in \mathcal{C}(K): x,y \in L\} \setminus \{K\}.
    \end{align*}
    is a \(K_\sigma\) set because \(\mathcal{C}(K)\) is compact. It follows from the Arsenin-Kunugui Theorem {\cite[Theorem 35.46]{Kechris}} that \(\pi_{1,2,3}(B)\) is Borel. But since $A$ is equal to \[\{(x,y,K)\in X^2\times \mathcal{C}(X):x,y\in K\} \setminus \pi_{1,2,3}(B),\] it is a Borel set.
    
     Now, let \(X_T\) be the union of all equivalence classes of \(E_X^C\) that contain a triod. There can be at most countably many such classes by \Cref{Theorem: pairwise disjoint family of triods in the plane is countable}. Since we assume that every equivalence class is Borel, we get that \(X_T\) is Borel. Set \(X_N = X \setminus X_T\) and consider 
    \[C = A\cap (X_N^2\times \mathcal{C}(X)).\] By \Cref{Lemma: condition for triod containing}, we get that for any \((x,y) \in X_N\times X_N\) the section \(C_{(x,y)}\) contains at most two continua and therefore \(\pi_{1,2}(C)\) is Borel by \Cref{Theorem: Lusin-Novikov}. We note that any two points contained in a continuum are contained in a continuum irreducible between them. Thus, we have \(\pi_{1,2}(C) = E_{X}^C|_{X_N}\) and so it is a Borel equivalence relation. Since \(X_T\) and \(X_N\) are both \(E_X^C\)-invariant, we have \(E_X^C = E_X^C|_{X_N} \cup E_X^C|_{X_T}\) and so \(E_X^C\) is a Borel equivalence relation. 
\end{proof}
Although there are parallels between the arc-connection and the continuum-connection relations, we do not know whether the continuum-connection relation on $G_\delta$ subsets of the plane is always Borel.

We end this section by a result that relates the product of the continuum-connection relation on a space with the relation on the product space just as before.

\begin{Proposition}
    Let \(X\) be a Polish space. Then \(E_{X^\N }^C\) is Borel isomorphic to \((E_X^C)^\N \), and hence \(E_{X^\N }^C \Bsim (E_X^C)^\N \). 
\end{Proposition}
\begin{proof}
    This again follows easily. Let \((x_n)_{n \in \N },(y_n)_{n \in \N } \in X^\N \). If \((x_n)_{n \in \N },(y_n)_{n \in \N }\) are contained in a continuum in \(X^\N \), then the points \(x_n \text{ and }y_n\) are contained in a continuum in \(X\) for every \(n \in  \N \) since projections of continua are again continua. On the other hand, if \(x_n \text{ and }y_n\) are contained in a continuum in \(X\) for every \(n \in \N \), then \((x_n)_{n \in \N } \text{ and } (y_n)_{n \in \N }\) are contained in a continuum in \(X^\N \) because product of countably many continua is a continuum.
\end{proof}

\section{Chain Continuum-connection Relation}

In this section, we present some results concerning the complexity of the chain continuum-connection relation. The property of being chain continuum-connected is sometimes referred as continuum-chainability, which we avoid to prevent any confusion with chainable continua. Some results on the chain continuum-connected continua can be found in \cite{On_weakly_chainable_continua, Lelek_weakly_chainable_continua}. We start with an auxiliary equivalence relation related to the chain-connection relation.

\begin{Definition}
    Let \(X\) be a Polish space and \(\varepsilon > 0\). We define the equivalence relation 
    \begin{align*}
        E_X^{CC,\varepsilon} = \{(x,y)\in &X\times X: \exists K_0, \dots, K_k \in \mathcal{C}_\varepsilon(X),\, k\in \N, \text{ such that} \\ & x \in K_0, \, y \in K_k \text{ and }K_i \cap K_{i+1} \neq \emptyset \text{ for } 0\leq i<k\}
    \end{align*}
    on $X$ where 
    \(\mathcal{C}_\varepsilon(X) = \{K \in \mathcal{C}(X): \diam(K) < \varepsilon\}\).
\end{Definition}
\begin{Remark*}
 The equivalence relations \(E_X^{CC,\varepsilon}\) and \(E_X^{CC}\) may depend on the choice of the metric on \(X\) for a general Polish space \(X\). On the other hand, \(E_X^{CC}\) does not depend on the metric put on a compact space \(X\). However, all the properties of the chain-connection relation that we show are independent of the metric, and therefore we do not specify any metric in the proofs. In other words, the results hold for any choice of compatible metric.
\end{Remark*}

We again start this section by showing that the chain-connection relation is very closely related to the arc-connection relation.
\newpage

\begin{Proposition}
\label{Proposition: partial cain connection is arc-connection on hyperspace of small continua}
    Let \(X\) be a Polish space and let \(\varepsilon >0\). Then \(E^{CC,\varepsilon}_X \Bsim E^A_{\mathcal{C}_\varepsilon(X)}\).
\end{Proposition}

\begin{proof}
    \((\leq _B)\) Consider the continuous mapping \(x \mapsto \{x\}\) defined from \(X\) to \(\mathcal{C}_\varepsilon(X)\). Let \(x,y\in X\). If \((x,y) \in E_{X}^{CC,\varepsilon}\), then the fact that \((\{x\},\{y\}) \in E^A_ {\mathcal{C}_\varepsilon(X)}\) can be shown in a similar way as in the proof of \Cref{Theorem: continuum-connection is acrwise connecion on hyperspace}. Now, suppose that there is a continuous function \(\varphi: [0,1] \to \mathcal{C}_\varepsilon(X)\) such that \(\varphi(0) = \{x\}\) and \(\varphi(1) = \{y\}\). Let $d$ be a compatible metric with the topology of $X$.

    We have \(\sup_{r \in [0,1]}\diam \varphi(r) < \varepsilon\) by compactness of \([0,1]\). Thus, there exists \(\eta > 0\) such that \(\sup_{r \in [0,1]}\diam \varphi(r) = \varepsilon - \eta\). Since \(\varphi\) is uniformly continuous, there is \(\delta > 0\) such that \(|r-s| < \delta\) implies \(d_H(\varphi(r),\varphi(s)) < \eta/2\). Let \(I\subseteq [0,1]\) be a closed interval of length less than $\delta$. Let  \(x_0,y_0\in \bigcup \varphi(I)\) such that \(x_0\in \varphi(r_0)\) and \(y_0\in \varphi(s_0)\) where \(r_0,s_0\in I\). Then since \(d_H\left(\varphi(r_0),\varphi(s_0)\right)<\eta /2\), there is \(x'\in \varphi(r_0)\) such that \(d(x',y_0)<\eta /2\). It follows that \(d(x_0,y_0)<d(x_0,x')+d(x',y_0)<(\varepsilon-\eta )+\eta /2<\varepsilon\), which shows that the diameter of the compact set \(\bigcup \varphi(I)\) is less than \(\varepsilon\).
   
    We can cover the interval \([0,1]\) by finitely many closed intervals \(I_0, \dots, I_n\) of length less than \(\delta\). Then \(\bigcup\varphi(I_0), \dots, \bigcup \varphi(I_n)\) is a chain of continua in \(\mathcal{C}_\varepsilon(X)\) connecting \(x\) and \(y\), where the argument that \(\bigcup \varphi(I_i)\) is a continuum is similar to that of \Cref{Theorem: continuum-connection is acrwise connecion on hyperspace}. 

\((\geq _B)\) A Borel selector function on \(\mathcal{C}_\varepsilon(X)\) works as in \Cref{Theorem: continuum-connection is acrwise connecion on hyperspace}.
\end{proof}

Since \(E_X^{CC} = \bigcap_{n \geq 1 }E_{X}^{CC,\frac{1}{n}}\), we have \(E_X^{CC} \Bsim  \bigcap_{n \geq 1 }E_{\mathcal{C}_{\frac{1}{n}}(X)}^A\) for any Polish space $X$ by the proposition above. We expressed the chain-connection relation as a countable intersection, and therefore the following reducibility result on the product of relations will be useful.

\begin{Lemma}
\label{Lemma: intersection of equivalence relations}
    Let \(X,Y\) be standard Borel spaces and let $E_n,F_n$ be Borel equivalence relations on $X$ and $Y$ for all $n\in \N $; respectively. Suppose that \(E_n \leq_B F_n\) for all \(n \in \N \). Then \(\displaystyle \bigcap_{n \in \N }E_n \leq_B \displaystyle\prod_{n\in \N }F_n\).
\end{Lemma}

\begin{proof}
   For any $n\in \N $, let \(f_n : X \to Y\) be a Borel reduction from \(E_n\) to \(F_n\). Then define \(f: X \to Y^\N \) as \(f(x) = (f_n(x))_{n \in \N }\). We show that \(f\) is the desired Borel reduction. The mapping $f$ can be easily seen to be Borel. Let \((x,y) \in \bigcap_{n \in \N }E_n\). For all $n\in \N $, we have \((f_n(x),f_n(y)) \in F_n\) as \(f_n\) is a Borel reduction. But then \((f_n(x),f_n(y))_{n \in \N } \) is in the product \(\prod_{n \in \N } F_n\). Conversely, suppose that \((f_n(x),f_n(y))_{n \in \N }\) is in the product \(\prod_{n \in \N } F_n\). Then since each \(f_n\) is a Borel reduction, we have \((x,y) \in E_n \) for every \(n \in \N \) implying that \((x,y) \in \bigcap_{n \in \N }E_n\).
\end{proof}

We shall use this result in the sequel. We continue our discussion by showing that there is an upper bound on the Borel complexity of \(E_X^{CC,\varepsilon}\) on locally compact Polish spaces.

\begin{Lemma}
\label{Lemma: epsilon arc-connection is K sigma}
    Let \(X\) be a locally compact Polish space and \(\varepsilon > 0\). Then \(E_X^{CC, \varepsilon}\) is a \(K_\sigma\) equivalence relation.
\end{Lemma}

\begin{proof}
    Since \(X\) locally compact and separable, we can find an increasing sequence \((G_n)_{n \in \N }\) of open sets in \(X\) such that \(\closure{G_n}\) is compact and \(X = \bigcup_{n \in \N }G_n\). For any $\varepsilon>0$, we set 
    \[E_n^{\varepsilon} = E_{\closure{G_n}}^{CC, \varepsilon} \cup \id (X).\]
 We show that \(E_n^{\varepsilon}\) is \(K_\sigma\) and that \(E_{X}^{CC,\varepsilon} = \bigcup_{n \in \N }E_n^\varepsilon\). Fix \(k \in \N \). Then 
 \begin{align*}
     \{(x,y)\in \closure{G_n}\times \closure{G_n}: \exists &K_0, \dots ,K_k \in \mathcal{C}(\closure{G_n}) \text{ of diameter }\leq \delta  \text{ such that }\\ & x \in K_0, \ y \in K_k  \text{ and } \ K_i \cap K_{i+1} \neq \emptyset \text{ for }0\leq i<k\}
\end{align*}
is a compact set for any \(\delta >0\) because all the relations defining the set are closed relations on compact sets. It follows that the set 
\begin{align*}
        E^\varepsilon_{n,k} = \{(x,y)&\in \closure{G_n}\times \closure{G_n}: \exists K_0, \dots K_k \in \mathcal{C}_\varepsilon(\closure{G_n}) \text{ such that }\\  & x \in K_0, \ y \in K_k \text{ and }K_i \cap K_{i+1} \neq \emptyset \text{ for }0\leq i<k\}
    \end{align*}
    is \(K_\sigma\) because it can be written as a countable union of the sets of the previous form. But then $E^{CC,\varepsilon}_{\closure{G_n}}=\bigcup _{k\in \N } E^\varepsilon_{n,k}$. This implies that $E^\varepsilon_n$ is a $K_\sigma$ set.

Since \(E_n^\varepsilon \subset E_{X}^{CC, \varepsilon}\) for all \(n \in \N \), we have \(\bigcup_{n \in \N }E_n^\varepsilon \subset E_{X}^{CC, \varepsilon}\). For the other inclusion, let \((x,y) \in E_X^{CC, \varepsilon}\). Then there is \(k\in \N\) and \(K_0, \dots,K_k \in \mathcal{C}_{\varepsilon}(X)\) such that \(x \in K_0\), \(y \in K_k\) and \(K_i \cap K_{i+1} \neq \emptyset\) for \( 0\leq i<k\). For any $i\in \{0,...,k\}$, we can find $n_i\in \N $ such that $K_i\subset G_{n_i}$ because $K_i$ is compact. Let \(m = \max _{0 \leq i \leq k}n_i\). Then we have \(K_i \in G_m\) for all  \(0\leq i \leq k\) and thus \((x,y) \in E_{m}^\varepsilon\subset \bigcup_{n \in \N }E_n^\varepsilon  \).

It follows that \(E_X^{CC,\varepsilon}\) is a \(K_\sigma\) equivalence relation.
\end{proof}

We can use this result together with the existence of a universal \(K_\sigma\) equivalence relation to bound the complexity of the chain continuum-connection relation on locally compact Polish spaces. 
\begin{Theorem}
\label{Theorem: chain continuum relation is smaller than l infinity to omega}
    Let \(X\) be a locally compact Polish space. Then \(E_X^{CC} \Bleq (\R/\ell_{\infty})^\N \).
\end{Theorem}
\begin{proof}
    We have \(E_{X}^{CC} = \bigcap_{n \geq 1}E_X^{CC, \frac{1}{n}}\). Moreover, \(E_X^{CC, \frac{1}{n}} \Bleq \R/\ell_\infty\) for all \(n\geq 1\) by \Cref{Lemma: epsilon arc-connection is K sigma} and \Cref{Theorem: l infinity is universal K sigma}. Thus, \(E_X^{CC} \Bleq (\R/\ell_{\infty})^\N  \) by \Cref{Lemma: intersection of equivalence relations}.
\end{proof}
We will show later that this upper bound is optimal. However, it turns out that the situation is a lot simpler in the plane. Let us first state a technical lemma showing that under some mild assumptions the chain-connection is the same as the arc-connection relation.

\begin{Lemma}
\label{Lemma: arc-connection is the same as chain continuum-connection without triod}
    Let \(X\) be a metric space. Suppose that for every \(n \in \N\), no equivalence class of \(E_X^{CC,\frac{1}{n}}\) contains a triod. Then \(E_X^{CC} = E_X^A\).
\end{Lemma}
\begin{proof}
    We need to show two inclusions.
    The inclusion \(E_X^{A} \subset E_X^{CC}\) holds trivially. To show the other inclusion, let \((x,y) \in E_X^{CC}\).  For every \(n \geq 1\), let \(C_n\) be the equivalence class of \(E_X^{CC, \frac{1}{n}}\) that contains \(x\) and \(y\). Then \(C_n\) does not contain a triod for any $n\geq 1 $ by the assumption. Moreover, the points \(x \text{ and }y\) are contained in at most two distinct continua in \(C_n\) that are irreducible between them by \Cref{Lemma: condition for triod containing}. Since we have \(C_{n+1} \subset C_n\) for every \(n\geq 1\), there exists a continuum \(K \subset X\) irreducible between \(x\) and \(y\), where \(K\) is contained in a \(\frac{1}{n}\)-chain of continua between \(x\) and \(y\) for every \(n \geq 1 \). We show that \(K\) is a Peano continuum and thus it is arc-wise connected. To achieve this, we use a standard characterization of Peano continua (see {\cite[Theorem 8.4]{Nadler_continuum_theory}}).

    Pick \(\delta > 0\) and let \(n \in \N \) such that \(\frac{1}{n} < \delta\). We can find \(k \in \N\) and \(K_0,...,K_k \in \mathcal{C}_{\frac{1}{n}}(X)\) such that \(x \in K_0\), \(y \in K_k\), \(K \subset \bigcup_{i=0}^{k}K_i\) and \(K_i \cap K_{i+1} \neq \emptyset\) for every \(0\leq i<k \). Since \(K \subset \bigcup_{i=0}^{k}K_i\) and \(\bigcup_{i=0}^{k}K_i\) is a continuum that does not contain any triod, we get that \(K \cap K_i\) has at most two components for every \(0\leq i\leq k\) by \Cref{Lemma: triodic number of comp.}. For every \(0\leq i\leq k\), let \(K^0_i,K^1_i \in \mathcal{C}_{\frac{1}{n}}(X)\) be such that \(K \cap K_i = K^0_i \cup K^1_i\). Then we have that 
    \[K = \bigcup_{i=0}^k \bigcup_{j=0}^1 K^j_i.\]
    Since \(\delta\) was arbitrary we get that \(K\) is a Peano continuum, and thus there is an arc in \(X\) connecting \(x\text{ and }y\). Hence \((x,y) \in E_X^{A}\). 
\end{proof}

\begin{Corollary}
\label{Corollary: for a space without riods chain connection is arc-connection}
    Let \(X\) be a Polish space. Suppose that \(X\) does not contain any triod. Then \(E_X^{CC} = E_X^A\).
\end{Corollary}
\begin{proof}
    This follows from \Cref{Lemma: arc-connection is the same as chain continuum-connection without triod}.
\end{proof}

\begin{Theorem}
\label{Theorem: chain continuum-connection is E_0 in the plane on locally compact sets}
 Let \(X\) be a locally compact subset of the plane. Then \(E_X^{CC} \Bleq E_0\).
\end{Theorem}
\begin{proof}
    By \Cref{Lemma: epsilon arc-connection is K sigma}, we have that \(E_X^{CC, \varepsilon}\) is Borel for every \(\varepsilon > 0\). For any \(n\geq 1\), let \(T_n\) be the union of all equivalence classes of \(E_X^{CC,\frac{1}{n}}\) that contain a triod. Since \(E_X^{CC,\frac{1}{n}}\) is Borel, then so is \(T_n\) by \Cref{Theorem: pairwise disjoint family of triods in the plane is countable}. We note that \(T_{n+1} \subset T_n\) for all \(n\geq 1\). Let \(N_n = X \setminus T_n\) and consider \(N = \bigcup_{n\geq 1 }N_n\) and \(T = \bigcap_{n \geq 1 } T_n\). Since \(T_n\) is \(E_X^{CC}\)-invariant for every \(n \geq 1\), we get that \(T\) and \(N\) are also \(E_X^{CC}\)-invariant. By \Cref{Lemma: arc-connection is the same as chain continuum-connection without triod}, we have that \(E_X^{CC}|_N = E^{A}_X|_N\).    

It follows that \(E_X^{CC}|_N = E^{A}_X|_N\Bleq E_0\) by \Cref{Theorem: Arc-connection is E_0 in the plane}. We will show that \(E_{X}^{CC}|_T\) is smooth, which suffices to conclude that \(E_{X}^{CC} \Bleq E_0\) as \(T\) and \(N\) are \(E_{X}^{CC}\)-invariant.

Without loss of generality, we may assume that each \(T_n\) is non-empty and has exactly countably infinite equivalence classes of \(E_{X}^{CC,\frac{1}{n}}\). Let \(R_n=\{x_n^{i}: i \in \N \}\) be a set that contains exactly one point from every equivalence class of \(T_n\) in \(E_{X}^{CC, \frac{1}{n}}\)  for all \(n\geq 1 \). Given \(x \in T\), let \(f(x)\) be the unique point in \(\prod_{n\geq 1 }R_n\) such that \((x,f(x)) \in E_{X}^{CC, \frac{1}{n}}\) for every \(n \geq 1 \). Then $f:T\rightarrow \prod_{n\geq 1 }R_n$ is a Borel mapping because the inverse image of a singleton is either empty or a $E^{CC}_X$-class. It is straightforward to check that \(f\) Borel reduces $T$ to the equality relation on the product. 
\end{proof}

We do not know whether \Cref{Theorem: chain continuum-connection is E_0 in the plane on locally compact sets} holds for a general planar \(G_\delta\) set. However, we note that in the proof above, we essentially used nothing but the fact that \(E_X^{CC,\varepsilon}\) is Borel for every \(\varepsilon>0\). Thus, for \(E_X^{CC}\) on a \(G_\delta\) subset \(X\) of the plane to be Borel reducible to \(E_0\), it would be enough to have \(E_X^{CC,\varepsilon}\) Borel for a sequence of positive real numbers converging to zero.

Even though the arc-connection and the chain-connection relations are not the same in general, they are closely related. In fact, there are some natural sufficient conditions on Polish spaces to ensure that \(E_X^{CC}\) and \(E_X^{A}\) are equal.

\begin{Lemma}
\label{Lemma: for a unicoherent space chain connection is arc-connection}
    Let \(X\) be a Polish space. Suppose that \(C_1 \cap C_2\) is connected for any \(C_1,C_2 \in \mathcal{C}(X)\). Then \(E_X^{CC} = E_X^A\).
\end{Lemma}
\begin{proof}
    The fact that \(E_X^{A} \subset E_X^{CC}\) is trivial. We will show the other inclusion. Let \((x,y) \in E_X^{CC}\). Then there is \(K \in \mathcal{C}(X)\) such that \(x,y \in K\). We can then find a continuum \(C \subset K\) containing \(x \text{ and }y \) that is irreducible between these points. By the assumption of the lemma, $C$ is the unique such continuum.

    We will show that \(C\) is a Peano continuum, and hence arc-wise connected. Let \(\varepsilon >0\). We can find \(K_0, \dots, K_k \in \mathcal{C}_\varepsilon(X)\) for some \(k \in \N\) such that \(x \in K_0\), \(y \in K_k\) and \(K_i \cap K_{i+1} \neq \emptyset\) for every \(0\leq i<k\). From the uniqueness of \(C\), we get \(C \subset \bigcup_{i=0}^{k}K_i\). By the assumption of the lemma, we have \(C \cap K_i \in \mathcal{C}_{\varepsilon}(X)\) for every \(0\leq i<k\). Thus \(C\) is covered by a finite number of continua of diameter smaller then \(\varepsilon\). Since \(\varepsilon\) was arbitrary, we conclude that \(C\) is a Peano continuum by a characterization of Peano continua \cite[Theorem 8.4]{Nadler_continuum_theory} and thus arc-wise connected. Thus \((x,y) \in E_X^{A}\). 
\end{proof}
\begin{Remark}
    We note that the assumptions of \Cref{Lemma: for a unicoherent space chain connection is arc-connection} are satisfied for instance for hereditarily unicoherent continua, circle-like continua or arc-like continua (see \cite[Theorem 12.1]{Nadler_continuum_theory}.
\end{Remark}

Again, for the chain-connection relation we can relate taking a product of the space to product of equivalence relations.

\begin{Proposition}
\label{Proposition: chain continuum-connection and products}
    Let \(X\) be a Polish space. Then \(E_{X^\N }^{CC}\) is Borel isomorphic to \((E_X^{CC})^\N \) and thus \(E_{X^\N }^{CC} \Bsim (E_X^{CC})^\N \).
\end{Proposition}
\begin{proof}   
    This follows from the following facts. First, taking product does not increase the diameter of a set too much and projections do not increase diameter too much. Second, products and projections of continua are again continua.
\end{proof}

The upper bound for the complexity of the chain-connection equivalence relation on locally compact Polish spaces can actually be realized by a continuum.
\begin{Example}
\label{Example: X with chain continuum l infinity to omega}
    There exists a continuum \(X\) such that \(E^{CC}_X \Bsim (\R/\ell_{\infty})^{\N }\).
\end{Example}
\begin{proof}
    We construct a continuum \(X\) such that \((\R / \ell_{\infty}) \Bleq E_{X}^{CC}\). This will be enough since we will have \((\R / \ell_{\infty})^\N  \Bleq E_{X^\N }^{CC}\) by \Cref{Proposition: chain continuum-connection and products} and \(E_{X^\N }^{CC} \Bleq (\R / \ell_{\infty})^\N \) by \Cref{Theorem: chain continuum relation is smaller than l infinity to omega}. 
    
    We shall construct the space as an inverse limit. Let \(V_0 = \{0\} \) and let \(V_{n+1} = V_n \times \{0, \dots, n+1\}\) for all \(n\in \N \). Let \(X_n\) be a connected graph with vertices in \(V_n\) and edges between any two distinct points \(x,y \in V_n\) satisfying \(\max_{0\leq i\leq n}|x(i)-y(i)| \leq 1\). By a graph, we mean a metric graph where all edges are isometric to \([0,1]\).

    For \(n \in \N \), let \(f_n:X_{n+1} \to X_n\) be defined as follows. If \(n = 0\), then \(f_0\) maps everything to the unique point of \(V_0\). For \(n \geq 1\), we define \(f_n\)\ on \(V_{n+1}\)  as the natural projection onto \(V_n\). On the edges, we have two possibilities. If the projections of the endpoints are different, \(f_n\) homeomorphically maps one edge onto the other. If the end points are mapped on the same point, then we pick any edge coming from the projection and \(f_n\) maps the edge between the points onto the part \([0,1/2]\) of the picked edge as \(t \mapsto (1-|2t-1|)/2\) for \(t \in [0,1]\). Now, let \[X = \lim_{\leftarrow}(X_n,f_n).\]

    Consider the following equivalence relation \(E\) on \(\prod_{n \in \N }\{0, \dots,n\}\) \[xEy \iff \exists N \in \N  \text{ such that } \sup_{n \in \N }|x(n)-y(n)| \leq N.\]

    By {\cite[Proposition 19]{Rosendal_cofinal_families_of_Borel_equivalences}} we have \(E \Bsim \R/\ell_{\infty}\). Thus, it is enough to show \(E \Bleq E^{CC}_X\). Let \[V = \lim_{\leftarrow}(V_n,f_n|_{V_{n+1}}).\] Then there is a natural  homeomorphism \(V \approx \prod_{n \in \N }\{0, \dots,n\}\), we will show that this homeomorphism is a Borel reduction from \(E\) to \(E_X^{CC}|_{V}\). From the construction of \(X\) if for \(x,y \in \prod_{n \in \N }\{0, \dots,n\} \approx V\) we have \(\sup_{n \in \N }|x(n)-y(n)| \leq 1\), then \(x\text{ and }y\) are connected by an arc. From this we have that if \(x,y \in \prod_{n \in \N }\{0, \dots,n\} \approx V\) such that \((x,y) \in E\) then \((x,y) \in E_{X}^A\) and thus also \((x,y) \in E_{X}^{CC}\). On the other hand, if \((x,y) \not \in E\) then any continuum in \(X\) that contains the points \(x,y\) is indecomposable and of diameter at least \(1/2\). Furthermore, from the construction, any continuum of sufficiently small diameter is arc-wise connected. Thus, if the points were connected by chains of continua of small enough diameter, they would have to be connected by an arc, which is not possible.
\end{proof}

As mentioned earlier, a natural question is whether some of the results obtained can hold in higher dimensions or for more general spaces. We already mentioned that there are examples of continua in \(\R^3\) that answer the question regarding higher dimensions for the arc-connection relation. Here, we will construct a Polish space that answers this question for all of the studied connection relations at once. In fact, this construction is heavily inspired by \cite{Kunen_Starbird_arc_components_in_metric_continua}. We will also show that even for \(F_\sigma\) subsets in the plane, none of the three connection relations is necessarily analytic.

\begin{Example}
\label{Example: X such that all the relations are analytic non-Borel}
    There exists a Polish space \(X \subset \R^3\) such that \(E_X^{C} = E_X^{CC} = E_{X}^{A}\) and all of these equivalence relations are analytic non-Borel subsets of \(X\times X\).
\end{Example}
\begin{proof}
    Let \(\mathcal{C}_0\) be the Cantor middle third  set in \([0,1]\) and let \(\mathcal{C} = \mathcal{C}_0 \times \mathcal{C}_0 \times \{0\} \subset \R^3\). By the standard properties of analytic sets, see \cite[Exercise 14.3]{Kechris}, we can find  a \(G_\delta\) set \(G \subset \mathcal{C}\) such that \(\pi_{1}(G)\) is an analytic non-Borel subset of \(\mathcal{C}_0\).

    For \(x,y \in \R^3\) denote the line segment connecting \(x\) and \(y\) by  \(A_{x,y}\) and call \(x_0 = (0,0,1)\). Let \(Z = \mathcal{C}_0 \times [-1,2] \times \{0\}\) and define
    
    \[Y = Z \cup \bigcup_{c \in \mathcal{C}}A_{c,x_0} \text{ and }X = Z \cup \bigcup_{c \in G}A_{c,x_0}\ .\]
    It is clear that \(X\) is a \(G_\delta\) subset of \(Y\). For any \(x \in \mathcal{C}_0 \times \{-1\} \times \{0\} \), we have \((x,x_0) \in E_X^A\) if and only if \(\pi_1(x) \in \pi_1(G)\) by the construction. Thus, \([x_0]_{E_X^A} \cap \mathcal{C}_0 \times \{-1\} \times \{0\}\) is an analytic non-Borel set and so \(E_X^A\) is a non-Borel analytic equivalence relation.

    We have that \(E_X^A = E_X^{CC}\) by \Cref{Lemma: for a unicoherent space chain connection is arc-connection}. Since any subcontinuum of \(Y\), and thus also of \(X\), is arc-wise connected, we get \(E_X^C=E_X^A\).
\end{proof}
\begin{Example}
    There exists an \(X \subset \R^2 \) an \(F_\sigma\) set such that \(E_X^A = E_X^{CC} = E_X^C\) and all of these equivalence relations are not analytic subsets of \(X\times X\).
\end{Example}
\begin{proof}
    Let \(\mathcal{C}\) be the Cantor middle-third set in \([0,1]\). Let \(A \subset \mathcal{C}\) be a coanalytic non-Borel set so that it is not analytic. Since \(A\) is coanalytic, we can find an \(F_\sigma\) subset \(B \subset \mathcal{C} \times \mathcal{C}\) such that \[A = \{x\in \mathcal{C}: \forall y \in \mathcal{C}\  \, (x,y) \in B\}\] by \cite[Section 32.B]{Kechris}. Now, set 
    \begin{align*}
        &X_0 = \mathcal{C} \times ([-1,2] \setminus \mathcal{C}) \text{ and } X_1 = [0,1] \times \{-1\}.
    \end{align*}
    It is straightforward to check that \(X_0\) and \(X_1\) are \(F_\sigma\) subsets of the plane. Define 
    \[X = X_0 \cup X_1 \cup B.\] 
    Then \(X\) is an \(F_\sigma\) subset of the plane. Let \(x = (0,-1)\). It is easy to see that \[A=[x]_{E^A_X} \cap (\mathcal{C} \times \{2\}) .\] If \(E^A_X\) was analytic, then $[x]_{E^A_X}$, and hence $A$, would be analytic as well. Thus, \(E^A_X\) is not analytic. The fact that \(E_X^{CC} = E_X^A\) again follows by \Cref{Lemma: for a unicoherent space chain connection is arc-connection}. For \(E_X^A = E_X^C\), it suffices to note that any subcontinuum of \(X\) is arc-wise connected.
\end{proof}
The main reason why the equivalence relations are not even analytic for some planar \(F_\sigma\) sets is that to study the equivalence relations on Polish spaces, we are heavily using the set of arcs or continua in the Polish space of all compact sets (or all continua). On the other hand, it may not be possible to establish a Polish space setting on an $F_\sigma$ subset unless we change the underlying topology in which case the arcs and continua of the subset fall apart. 

Although in  higher dimensions the equivalence relations are not Borel in general, we can still bound the complexity of the connection relations for some special classes of spaces.

\begin{Corollary}
    Let \(X\) be a tree-like continuum. Then \(E_X^A \Bleq (\R/\ell_{\infty})^\N \).
\end{Corollary}
\begin{proof}
    Since \(X\) is a tree-like continuum, it is hereditarily unicoherent. By \Cref{Lemma: for a unicoherent space chain connection is arc-connection} we have that \(E^A_X = E^{CC}_X\). So we get \(E_X^A \Bleq (\R/\ell_{\infty})^\N \) by \Cref{Theorem: chain continuum relation is smaller than l infinity to omega}.
\end{proof}

It turns out that we can show even more for the class of tree-like continua. 
\begin{Proposition}
\label{Proposition: bound for tree-like continua}
    There exists a Borel equivalence relation \(E\) such that \(E^A_X \leq_B E\) for every tree-like continuum \(X\) and \(E \leq_B E^A_T\) for some tree-like continuum \(T\).
\end{Proposition}
\begin{proof}
    By \cite{McCord_universal_P_like_compacta}, there exists a universal tree-like continuum \(T\). We will show that \(E^A_T\) is the desired equivalence relation. Let \(X\) be any tree-like continuum and \(\varphi: X \to T\) be an embedding. Clearly, \((x,y) \in E^A_X\) implies \((\varphi(x),\varphi(y)) \in E^A_T\). On the other hand, suppose that \((\varphi(x), \varphi(y)) \in E^A_T\) and let \(\psi\) be the arc connecting \(\varphi(x)\) and \(\varphi(y)\) in \(T\). If \(\psi\) is not contained in \(\varphi(X)\), then the space \(\varphi(X) \cup \psi\) is a subcontinuum of \(T\) that is not unicoherent, which is a contradiction since every tree-like continuum is hereditarily unicoherent. The second part clearly follows.
\end{proof}

\section{Questions}
In this section, we consider some open questions that arise from this investigation. Since we were able to show some of the results only for locally compact Polish spaces, a natural question is whether they hold in the greatest possible generality. In light of \Cref{Theorem: continuum-connection is Borel iff all classes are Borel}, we can formulate one for the continuum-connection relation as follows.
\begin{Question}
\label{Question: 1}
   Is $E^C_X$ Borel for any $G_\delta$ subset $X\subset \mathbb{R}^2$?
\end{Question}

If the answer to the question is positive, we can then further ask if the complexity of the continuum-connection relation is bounded for planar spaces.
\begin{Question}
\label{Question: 2}
    Is $E^C_X$ hypersmooth for any $G_\delta$ subset $X\subset \mathbb{R}^2$?
\end{Question}
Analogous questions can be asked about the chain continuum-connection relation.
\begin{Question}
\label{Question: 3}
   Is $E^{CC}_X$ Borel for any $G_\delta$ subset $X\subset \mathbb{R}^2$?
\end{Question}
Although a positive answer to this question is not enough to have any reasonable upper bound for the complexity of the chain-connection relation, we suspect that a positive answer to the question will also yield a positive answer to the next question (see the discussion after \Cref{Theorem: chain continuum-connection is E_0 in the plane on locally compact sets}).
\begin{Question}
\label{Question: 4}
    Is it true that \(E_X^{CC}\Bleq E_0\) for any $G_\delta$ subset $X\subset \mathbb{R}^2$?
\end{Question}

One can study some special classes of spaces and wonder the exact behaviour of the equivalence relations in that particular setting. Thus, the following is also a valid question.
\begin{Question}
    What is the exact complexity of \(E_T^A\), where \(T\) is a universal tree-like continuum?
\end{Question}
\begin{Remark*}
    We note that a universal tree-like continuum is not unique; however, the complexity of the arc-connection relation is essentially the same on any universal tree-like continuum.
\end{Remark*}

\bibliographystyle{abbrv}
\bibliography{bibliography}
\end{document}